\documentclass{article} 
\usepackage{graphicx}
\usepackage{subfigure}
\usepackage{mathtools}
\usepackage{enumerate}
\usepackage[utf8]{inputenc}
\usepackage{float}
\usepackage[labelformat=parens,labelsep=quad,skip=3pt]{caption}
\usepackage{amsmath}
\usepackage{dsfont}
\usepackage{tikz}
\usepackage{epstopdf}
\usepackage{pgfplots}
\usepackage{amssymb}
\usepackage{psfrag}
\newtheorem{theorem}{Theorem}
\newtheorem{proof}{Proof}
\newtheorem{definition}{Definition}
\newtheorem{example}{Example}
\newtheorem{Corollary}{Corollary}
\begin{document}
\title{Vector Contraction Analysis for Nonlinear Dynamical Systems }
		\author{Bhawana Singh, Debdas Ghosh, Shyam~Kamal, Sandip~Ghosh \\and Antonella Ferrara
			\date{}	
			\thanks{~~~Bhawana Singh, Shyam Kamal and Sandip Ghosh are from the Department of Electrical Engineering, Indian Institute of Technology (BHU), Varanasi, U.P., India e-mail: bhawanasingh.rs.eee17@iitbhu.ac.in, shyamkamal.eee@iitbhu.ac.in, sghosh.eee@iitbhu.ac.in, Debdas Ghosh is from the Department of Mathematical Sciences, Indian Institute of Technology (BHU), Varanasi, U.P., India email: debdas.mat@iitbhu.ac.in and Antonella Ferrara is from the Department of Industrial and Information Engineering, University of Pavia, Pavia, Italy email:antonella.ferrara@unipv.it}
			}
			\maketitle

		\begin{abstract}
			This paper derives new results for the analysis of nonlinear systems by extending contraction theory in the framework of vector distances. A new tool, vector contraction analysis utilizing a notion of the vector-valued norm which evidently induces a vector distance between any pair of trajectories of the system, offers an amenable framework as each component of vector-valued norm function satisfies fewer strict conditions as that of standard contraction analysis. Particularly, every element of vector-valued norm derivative need not be strictly negative definite for convergence of any pair of trajectories of the system. Moreover, vector-valued norm derivative satisfies a componentwise inequality employing some comparison system. In fact, the convergence analysis is performed by comparing the relative distances between any pair of the trajectories of the original nonlinear system and the comparison system. Comparison results are derived by utilizing the concepts of quasi-monotonicity property of the function and vector differential inequalities. Moreover, the results are also derived in the framework of the cone ordering instead of utilizing componentwise inequalities between vectors. In addition, the proposed framework is illustrated by examples.
		\end{abstract}
	\section{Introduction}
	One of the most fundamental problem in control theory is the analysis of the stability of dynamical systems. The major contribution in this field is done by Lyapunov \cite{jurdjevic1978controllability}. Lyapunov stability is used as a vehicle to convert a given composite differential system to a much simpler system. Lyapunov function describes the distance of the motion space from the origin \cite{lakshmikantham2013vector}. The major advantage of the Lyapunov's second method is that it does not need the knowledge of solutions of the differential equation and thus has wide applications but there is no general procedure for constructing Lyapunov function candidate for the stability analysis of nonlinear systems. Also, there are certain strict conditions of Lyapunov function candidate to be positive definite and its derivative to be negative definite for analyzing the asymptotic stability of dynamical systems about the equilibrium point.\\

	In order to simplify the Lyapunov function construction and relax the restriction for the asymptotic stability analysis, many researchers turn to vector Lyapunov functions as a substitute to standard Lyapunov functions. Vector Lyapunov function has been recognized as a more reliable tool than scalar Lyapunov function in perusing stability of dynamical systems. Bellman \cite{bellman1962vector} introduced the concept of vector Lyapunov function and further it is developed in \cite{siljak1983complex,martynyuk1998stability,martynyuk2003qualitative} exploiting their advantage for stability analysis of large-scale systems about the equilibrium point. Various research has been done on reaction networks in the past \cite{van2013mathematical}. It offers a flexible framework as every component of vector Lyapunov function satisfies less strict conditions as that of standard Lyapunov theory. Moreover, vector Lyapunov function derivative satisfies a component-wise inequality consisting of some comparison system \cite{nersesov2006stability}. One of the major restriction of Lyapunov stability analysis is that it is used for the systems having some specific attractor which is the weaker notion of stability than incremental stability. \\

	Contraction theory which is first popularized in \cite{lohmiller1998contraction} is also an incremental form of stability \cite{jouffroy2010tutorial}, that is the convergence of the trajectories of the system with respect to one another. In this theory, the convergence analysis is performed with the dynamics of the system in the differential framework. Numerous practices of contraction analysis have been done in different frameworks using Finsler distances \cite{forni2014differential}, Riemannian distances \cite{simpson2014contraction}, contraction metrics \cite{manchester2017control}\cite{aylward2006algorithmic}, matrix measures \cite{sontag2010contractive}\cite{aminzarey2014contraction}\cite{zahreddine2003matrix}, matrix measures for switched systems \cite{lu2016contraction}\cite{fiore2016contraction}, etc. \\ 
	
	Despite the extensive research in this area,  we notice that for the contraction analysis, the literature keeps using a scalar-valued relative distance between the trajectories and the negative definiteness property of the Jacobian. The use of scalar-valued distance makes the  contraction analysis restrictive, particularly for a large-scale system. Since it is complicated to check the negative definiteness, or its mild variations (see \cite{bellman1962vector}), of the Jacobian for a large-scale system.  To overcome this problem,  a  vector-valued  contraction analysis is explored in the present paper.  \\

	In this paper, we develop vector contraction analysis utilizing vector distance between any pair of trajectories of the system to conclude convergence. This approach relaxes the negative definite derivative condition of standard contraction analysis. Specifically, the convergence analysis is observed with the help of \emph{comparison system} by comparing the relative distances of the trajectories of the original nonlinear dynamical system and the comparison system. In fact, we propose some comparison results which connects the solutions of the original vector system and the auxiliary system by employing quasi-monotonicity property of the function with the help of a notion of the vector-valued norm. Moreover, the results are also derived in the framework of the cone ordering to remove the component-wise comparison of vectors. Consequently, from the derived results, one can conclude that the convergence between trajectories of the \emph{comparison system} implies the convergence between trajectories of the original dynamical system.\\

	The rest of the paper goes ahead with the notations and preliminaries consisting required definitions and notations  in Section 2. The convergence analysis theory via vector contraction analysis giving a detailed account of the vector-valued norm and main comparison results in the framework of vector inequalities and cone are illustrated in Section 3. As an illustration, examples are treated in Section 4.  Finally, brief conclusions end the paper.

	\section{Preliminaries and notations}
	We use the following notations throughout the paper.

	\begin{itemize}
		\item $\mathbb{R}$ denotes the set of real numbers.
		\item $\mathbb{R}_{+}$ denotes the set of non-negative real numbers.
		\item $\mathbb{R}^{n}$ denotes the set of $n\times1$ column vectors.
		\item We denote 
		\begin{align*}
		\mathbb{R}_+^n := \{x = (x_1, x_2, \cdots, x_n)^T   \in \mathbb{R}^n ~|~  & x_i \ge 0, \\ 
		& i = 1, 2, \cdots, n\}.
		\end{align*}
		\item We use $x \le y$, for $x = (x_1, x_2, \cdots, x_n)^T$ and $y = (y_1, y_2, \ldots, y_n)^T$, if $x_i \le y_i$ for each $i = 1, 2, \dots, n$. 
		\item We use $x<y$, for $x=(x_1, x_2, \cdots, x_n)^T$ and $y=(y_1, y_2, \ldots, y_n)^T$, if $x_i < y_i$ for each $i= 1, 2, \dots, n$. 
		\item The notation $\langle x, y \rangle$, for $x, y$ in $\mathbb{R}^{n}$, denotes the usual inner product $x^T y$.
		\item $\|x\|$ is the usual Euclidean norm of $x$ in $\mathbb{R}^{n}$.
		\item $\|\delta x\|_v$ is a vector-valued norm of $\delta x \in \mathbb{R}^n$  as defined in the equation (\ref{vect-norm}). 
		\item For a vector $x = (x_1, x_2, \cdots, x_n)^T\in \mathbb{R}^n$, we denote the the diagonal matrix $\text{diag}(x_1, x_2, \cdots, x_n)$ by $\text{diag}(x)$. 
		\item Let $d = (d_1, d_2, \cdots, d_n)^T \in \mathbb{R}^n$. Corresponding to the diagonal matrix $\text{diag}(d_1, d_2, \cdots, d_n)$, the vector $d$ is denoted by $\text{dvec}(\text{diag}(d))$. 
		\item Let $A\subseteq \mathbb{R}^k$ and $B \subseteq{R}^p$ be two nonempty sets. The set of all continuous functions from $A$ to $B$ is denoted by $C[A, B]$. 
	\end{itemize}

	\begin{definition}
		(Cone \cite{leela1976cone}). A nonempty set $K \subset \mathbb{R}^n$ is called a cone if for each $x$ in $K$ and a nonnegative scalar $\lambda$, the vector $\lambda x$ is in $K$. 
	\end{definition}

	In the rest of the article, we \emph{assume} that any cone $K$ under consideration possesses the following properties:
	\begin{enumerate}[(i)]
		\item $K$ is a closed and convex set,
		\item $K \cap (-K) = \{0\}$, and 
		\item $K^{\mathrm{o}}$, the interior of $K$, is nonempty.
	\end{enumerate}
	It is to be noted that a cone $K$ induces a partial order relation on $\mathbb{R}^n$ defined by
	\[x \le_{K} y ~ \Longleftrightarrow ~ y - x \in K.\]
	
	The \emph{adjoint cone} of a cone $K$, denoted $K^*$, is defined by
	$K^* = \{\phi \in \mathbb{R}^n | \langle \phi, x \rangle \ge 0 \text{ for all } x \in K\}.$
	
	It is noteworthy that if $\partial K$ denotes the boundary of the cone $K$, and $K_0 = K \setminus \{0\}$, then (see \cite{leela1976cone}),
	\[x \in \partial K ~ \Longleftrightarrow ~ \langle \phi, x \rangle = 0 ~\text{ for some }~ \phi \in K_0^*.\]
	
	\begin{definition}\label{quasimonotone_wrt_cone}
		(Quasi-monotone function relative to a cone \cite{leela1976cone}).
		Let $A$ be a nonempty subset of $\mathbb{R}^n$. A function $F \in C[A, \mathbb{R}^n]$ is called quasi-monotone, on $A$, relative to a cone $K$ if
		$x, y \in A $ with $y - x \in \partial K \Longrightarrow$ there exists $\phi \in K_0^*$ such that
		$\langle \phi, y -x \rangle = 0$ and  $\langle \phi, F(y) - F(x) \rangle \ge 0 $.
	\end{definition}
	If $K = - \mathbb{R}^n$, the non-positive orthant of $\mathbb{R}^n$, then the partial order relation $x \le_{K} y$ reduces to the usual component-wise ordering and the Definition \ref{quasimonotone_wrt_cone} reduces to the following for the non-decreasing case.

	\begin{definition}
		({Quasi-monotone non-decreasing function} \cite{lakshmi1969differential}).
		Let $A$ be a nonempty subset of $\mathbb{R}^n$ and $x = (x_1, x_2, \cdots, x_n)^T$ be a generic element of $A$. A function $F=(F_1, F_2, \cdots, F_n)^T \in C[A, \mathbb{R}^n]$ is called quasi-monotone non-decreasing on $A$ if for each $i \in \{1, 2, \cdots, n\}$,
		$F_i \text{ is nondecreasing in } x_j \text{ for all } j = 1, 2,\cdots, i-1, i+1,\cdots,
		n.$
	\end{definition}

	\subsection{Contraction analysis}
	Consider the differential system:
	\begin{equation}\label{diff_syst}
	\Sigma :~\dot x=f(t,x) , \ \ x(t_{0})=x_{0}, \  t_{0}\geq0, \\
	\end{equation}
	where $f: \mathbb R_{+}\times\mathbb R^{n} \rightarrow \mathbb R^{n}$ is continuously differentiable, $x \in \mathbb{R}^n$ and $t$ is the time.
	We assume that the system $\Sigma$ has a unique solution $\psi(x_{0},t)$. 
	Let $\delta x $ be the virtual displacement between two neighboring trajectories of the vector field of (\ref{diff_syst}). Then the squared distance between these two trajectories is $\|\delta x \|^2 = {\delta x}^T \delta x$ (see \cite{lohmiller1998contraction}). The essential task of the contraction analysis for the system (\ref{diff_syst}) is to analyze the convergence of its solutions with the help of the virtual displacement $\delta x$. \\

	As $f$ is continuously differentiable, at a fixed time $t$, the exact differential of the system $\Sigma$ yields the variational system as
	\begin{equation}\label{diff_sys}
	d\Sigma:~\delta\dot x = \frac{\partial f(t,\psi(t,x_0))}{\partial x} \delta x { = h(\delta x,x,t)} 
	\end{equation}
	where { $h: \mathbb R^n\times\mathbb R^{n}\times\mathbb R_{+} \rightarrow \mathbb R^{n}$,} $\delta x \in \mathbb{R}^n$ and $t$ is time. To denote the solution to $d\Sigma$, we use $\delta \psi(t,x_0,\delta x_0)$ from the initial state $\delta x_0$ at time t along $\psi(t,x_0)$.
	Thus, the rate of change of the squared virtual distance ${\delta x}^T \delta x$ is given by
	\begin{equation}
	\frac{d}{dt}({\delta x}^{T}\delta x)=2{\delta x}^{T}{\delta \dot x}=2{\delta x}^{T}\frac{\partial f}{\partial x}\delta x.
	\end{equation}

	{ Let $\lambda_{\max} (x,t)$} be the largest eigenvalue of the symmetric part of the Jacobian $\frac{\partial f}{\partial x}$, { i.e., of $\tfrac{1}{2}\left( \frac{\partial f}{\partial x} + (\frac{\partial f}{\partial x})^T \right)$}. Then  any infinitesimal distance $\|\delta x\|$ converges exponentially to zero as $t \rightarrow \infty$, if $\lambda_{\max} (x,t)$ is uniformly negative definite (see \cite{lohmiller1998contraction}).


	\section{Convergence analysis via vector contraction analysis}
	We note that the derivative of the squared distance between a pair of neighboring trajectories of the system (\ref{diff_syst}) need not be strictly negative definite. In such cases, the convergence analysis is observed with the help of a \emph{comparison system}. Vector contraction analysis performs the convergence by comparing the relative distances of the trajectories of the original nonlinear dynamical system and the comparison system.
	In this study, we propose a few comparison results towards convergence analysis with the help of a notion of \textit{vector-valued norm} as defined below.

	\subsection{A vector-valued norm}
	We define a vector-valued norm as a function $\| \cdot \|_v: \mathbb{R}^n \rightarrow \mathbb{R}^m$ by
	\begin{equation}\label{vect-norm}
	\|\delta x \|_v = \sqrt{A~ \text{dvec}((\text{diag}(\delta x))^2)},
	\end{equation}
	where $A$ is a real matrix $\left( a_{ij} \right)_{m \times n}$ with all $a_{ij}$ nonnegative. Note that for $\delta x = (\delta x_1, \delta x_2, \cdots, \delta x_n) \in \mathbb{R}^n$, 
	
	\[(\text{diag}(\delta x))^2  = 
	\begin{bmatrix}
	\delta x_1^2 & 0 & \cdots & 0 \\
	0 & \delta x_2^2 & \cdots & 0 \\
	\vdots  & \vdots  & \ddots & \vdots  \\
	0 & 0 & \cdots & \delta x_n^2  
	\end{bmatrix}\]
	and hence 
	\[\text{dvec}(\text{diag}(\delta x))^2  = 
	\begin{bmatrix}
	\delta x_1^2 \\
	\delta x_2^2 \\
	\vdots \\
	\delta x_n^2  
	\end{bmatrix}.\] 
	Thus, explicitly, (\ref{vect-norm}) can be written as 
	{ 
		\[
		\|\delta x \|_v =
		\left[
		\begin{array}{c}
		D_1(\delta x) \\
		D_2(\delta x) \\
		\vdots \\
		D_m(\delta x)
		\end{array}
		\right]
		\coloneqq
		\left[
		\begin{array}{c}
		\sqrt{\sum_{j = 1}^n a_{1j} \delta x_j^2} \\
		\sqrt{\sum_{j = 1}^n a_{2j} \delta x_j^2 } \\
		\vdots \\
		\sqrt{\sum_{j = 1}^n a_{mj} \delta x_j^2}
		\end{array}
		\right].
		\]
		
		It is important to mention here that the terminology `vector norm' is often used in the literature (see \cite{aminzarey2014contraction}) to mean the `norm of a vector'. Thus, `vector norm' gives a \emph{scalar-valued norm}. However, note that the Definition \ref{vect-norm} introduces \emph{vector-valued norm} for a vector.
	}

	The following points are observed from the definition of vector-valued norm $\|\cdot\|_v$.
	\begin{enumerate}[(i)]
		\item Every component of $\|\delta x \|_v$, i.e., $D_i(\delta x)$ for each $i = 1, 2, \cdots, m$, is a (scalar) norm in $\mathbb{R}^n$ (for proof, see the proof of Theorem \ref{thm1} in Appendix).
		\item $\|\delta x\|_v$ reduces to a (scalar-valued) norm in $\mathbb{R}^n$ when $A$ is a matrix of order $n \times 1$.
		\item $\|\delta x\|_v$ reduces to the usual Euclidean norm in $\mathbb{R}^n$ when $A$ is a matrix of order $n \times 1$ with all entries $1$.
	\end{enumerate}
	These three points show that the vector-valued norm $\|\delta x\|_v$ is a true generalization of the notion of norm in $\mathbb{R}^n$.

	
	In the Theorem \ref{thm1}, we show that $\|\cdot\|_v$ follows all the properties of a \emph{norm} and $\|\cdot\|_v^2$ possesses the \emph{convexity} and \emph{locally-Lipschitzian} properties. We further show that $\| \cdot \|_v^2$ is differentiable and then compute its derivative.
	\begin{theorem}\label{thm1}
		\begin{enumerate}[(a)]
			\item (Norm property). \\ 
			The vector-valued norm $\| \cdot \|_v : \mathbb{R}^n \rightarrow \mathbb{R}^m$ defined in (\ref{vect-norm}) has the following properties
			\begin{enumerate}[(i)]
				\item $\|\delta x \|_v = 0 \in \mathbb{R}^m$ iff $\delta x = 0 \in \mathbb{R}^n$,
				\item $\|c \delta x\|_v = |c|~ \|\delta x \|_v$ for any $c \in \mathbb{R}$ and $\delta x \in \mathbb{R}^n$, and
				\item $\|\delta x + \delta y \|_v \le \| \delta x \|_v + \| \delta y \|_v$ for all $\delta x, \delta y \in \mathbb{R}^n$. \\ 
			\end{enumerate}

			\item (Convexity property). \\ 
			Let $F(\delta x) = \|\delta x\|_v^2$, $\delta x \in \mathbb{R}^n$. Then,
			for any $\delta x, \delta y \in \mathbb{R}^n$ and $\lambda \in [0, 1]$,
			\[F \left( \lambda \delta x + (1 - \lambda) \delta y \right) ~ \le ~ \lambda F(\delta x) + (1 -\lambda)F(\delta y). \]

			\item (Locally-Lipschitzian property). \\ 
			Let $F(\delta x) = \|\delta x\|_v^2$, $\delta x \in \mathbb{R}^n$. Then,
			for any compact set $C \subset \mathbb{R}^n$ there exists a constant $k \in \mathbb{R}$ such that
			\[\| F(\delta x) - F(\delta y) \| \le k \|\delta x - \delta y\|^2 ~\text{ for all }~ \delta x, \delta y \in C. \]

			\item (Differentiability). \\ 
			The function $F(\delta x) = \|\delta x\|_v^2$ for $\delta x \in \mathbb{R}^n$ is (Fr\'echet) differentiable in $\mathbb{R}^n$ and for any $\delta x \in \mathbb{R}^n$, { the Fr\'echet derivative of $F$ is } 
			\[F'(\delta x) = 2~A~\delta x. \]
			
		\end{enumerate}
	\end{theorem}
	
	\begin{proof}
		Proof is given in the Appendix.
	\end{proof}
	The notion of  norm $\|\cdot \|_v$ defined by (\ref{vect-norm}), evidently, induces a vector distance between a pair of points $x,~ x+\delta x\in\mathbb R^n$  as follows:
	$$\|x+\delta x-x\|_v=\sqrt{A~ \text{dvec}(\text{diag}(\delta x))^2}. $$
	{In the rest of the article, we assume that $A$ is a nonzero matrix, as the case of $A$ being zero matrix is uninteresting for the relative distance of two trajectories.} With the help of this distance, a few comparison results are shown in the Section \ref{com-results}.

	\subsection{Main comparison results}\label{com-results}
	In this section, we derive some results on the comparison of solutions for the system (\ref{diff_sys}) and the solution of an auxiliary (comparison) system $ \dot u = \phi(t,u)$, where $\phi$ possesses certain quasi-monotone property. 
	\begin{theorem}\label{theorem1}
		Consider the system (\ref{diff_sys}) as a linear system and a function $\phi \in C\left[\mathbb{R}_+ \times \mathbb{R}^n,~ \mathbb{R}^n\right]$, $(t, u) \mapsto \phi(t,u)$, which is quasi-monotone non-decreasing in $ u \in \mathbb{R}^n$.
		Suppose for any solution $\delta\psi(t,\delta x_0)$ on $t \ge t_0$ of the system (\ref{diff_sys}),
		\begin{align*} 
		& 2 A~ \text{dvec}(\text{diag}(\delta \psi) ~ \text{diag}(h(\delta \psi, t))) \\ 
		& ~<~ \phi\left(t, A~ \text{dvec}(\text{diag}(\delta \psi(t))^2\right) ~\text{ for all }~ t \ge t_0,   
		\end{align*}
		for a matrix $A = (a_{ij})_{n \times n}$ with $a_{ij} \ge 0$ for all $i, j = 1, 2, \cdots, n$. Further, let for $t \ge t_0 > 0$ there exists
		\footnote{The existence of a maximal solution is evident from Theorem 1.3.1 of \cite{lakshmi1969differential}.}
		a maximal solution
		\footnote{A solution $R(t)$ of the system (\ref{diff_syst}) is called a maximal solution if for every solution $x(t)$ of (\ref{diff_syst}), $ x(t) \le R(t)$ for all $t \ge t_0$.}
		$R(t)$ of 
		\begin{equation}\label{comparison}
		\dot u = \phi(t, u),~  u(t_0) = u_0 \ge 0.
		\end{equation}
		Then, any solution $\delta \psi(t,\delta x_0)$ of (\ref{diff_sys}) on $t\ge t_0$ which satisfies  \[A ~ \text{dvec}(\text{diag}(\delta x_0)^2) < u_0\] has the property that
		\[A ~ \text{dvec}(\text{diag}(\delta \psi(t))^2) < R(t) ~\text{ for all }~ t \ge t_0. \]
		From the above, the following conclusion holds: \\ 
		if $R(t) \rightarrow 0$ as $t \rightarrow \infty$, then $\|\delta \psi(t)\|\rightarrow 0$ as $t \rightarrow \infty$, which implies that all trajectories of the original dynamical system converges with respect to one another as $t \rightarrow \infty$.
	\end{theorem}
	
	\begin{proof}
		Consider the function \[D (t) \coloneqq A~ \text{dvec}(\text{diag}(\delta \psi(t))^2) \text{ for } t \ge t_0.\] 
		Let the component functions of $D(t)$ and $R(t)$ be $D_i(t)$ and $R_i(t)$, respectively, $i = 1, 2, \cdots, n$. 
		Evidently, if the $i$-th row of the matrix $A$ be $a_i^T$, then \[D_i(t) = a_i^T ~ \text{dvec}(\text{diag}(\delta \psi(t))^2)\] 
		for each $i = 1, 2, \cdots, n$. \\ 
		
		As $D(t_0) = A~\text{dvec}(\text{diag}(\delta \psi(t_0))^2) <  u_0 = R(t_0)$, and $D(t)$ and $R(t)$ are two continuous functions, there exists $\delta_1 > 0$ such that
		\[D(t) < R(t) ~\text{ for all }~ t \in [t_0, t_0 + \delta_1).\]  

		Construct a set
		\begin{equation}\label{S_construction} 
		S \coloneqq \bigcup_{i = 1}^n \left\{ t \in [t_0, \infty) \middle| R_i(t) \le D_i(t) \right\}. 
		\end{equation}
		We prove that $S$ is an empty set. Then, the proof will be complete. \\ 

		If possible let $S$ be not empty. Then, $S$, being a nonempty and bounded below set, has an infimum. Let $\tau = \inf S$. 
		We note that the set $S$ is closed, since $R(t)$ and $D(t)$ are continuous function on $[t_0, \infty)$. Therefore, $\tau \in S$ and hence there exists $j$ in $\{1, 2, \cdots, n\}$ such that  $D_j(\tau) = R_j(\tau)$. Moreover, \[a_i^T~ \text{dvec}(\text{diag}(\delta \psi(\tau))^2)= D_i(\tau) \ge R_i(\tau)\]  for all $i = 1, 2, \cdots, j-1, j+1, \cdots, n$. Therefore, due to quasi-monotone non-decreasing property of the function $\phi$, we obtain
		\begin{equation}\label{contradict_eqn1}
		\phi_j\left( \tau, A~ \text{dvec}(\text{diag}(\delta \psi(\tau))^2)\right) \ge \phi_j(\tau, R(\tau)).
		\end{equation}
		Again, since $D(t_0) < R(t_0)$, we have $\tau \neq t_0$ and hence $\tau > t_0$. \\ 
		By the definition of $\tau$, there exists $\delta_2 > 0$ such that $D_j(t) < R_j(t)$ for all $t \in (\tau - \delta_2, \tau) \subset [t_0, t_0 + \delta_1)$. Therefore,
		\begin{align}
		&~ \dot{D}_j(\tau) = \lim_{\xi \rightarrow 0 -} \frac{D_j(\tau + \xi) - D_j(\tau)}{\xi}  \notag \\
		~&~~\ge\lim_{\xi \rightarrow 0 -} \frac{R_j(\tau + \xi) - R_j(\tau)}{\xi} = \dot{R}_j(\tau) \notag \\
		\text{or},  &~ 2 ~ a_j^T ~ \text{diag}(\delta \psi(\tau)) ~ \text{diag}(h(\tau, \delta \psi(\tau))) \ge \phi_j(\tau, R(\tau)).
		\end{align}
		By the assumption 
		\begin{align*} 
		& 2 A~ \text{dvec}(\text{diag}(\delta \psi) ~ \text{diag}(h(\delta \psi, t))) \\ 
		& ~<~ \phi\left(t, A~ \text{dvec}(\text{diag}(\delta \psi(t))^2\right) ~\text{ for all }~ t \ge t_0,   
		\end{align*} 
		we obtain
		\begin{align*}
		\phi_j(\tau, R(\tau)) & ~>~ 2 ~ a_j^T ~ \text{diag}(\delta \psi(\tau)) ~ \text{diag}(h(\tau, \delta \psi(\tau))) \\
		& ~>~ \phi_j(\tau, R(\tau)), 
		\end{align*}
		which is a contradiction.
		Hence, the set $S$ is empty, and therefore for all $t \ge t_0$, 
		\begin{align}
		& R_i (t)> a_i^T \text{dvec}(\text{diag}(\delta \psi(t))^2)\text{ for all } i = 1, 2, \cdots, n, \notag \\ 
		\text{i.e., } & A~ \text{dvec}(\text{diag}(\delta \psi(t))^2) < R(t). \label{ab2}
		\end{align}
		Hence, the conclusion follows from (\ref{ab2}).
	\end{proof}
	In some cases, the estimation of derivative of $\|\delta x\|_{v}^{2}$ as a function of $t,\psi(t)$ and $A ~ \text{dvec}(\text{diag}(\delta \psi(t))^2)$ is more natural usually in nonlinear case. The following corollary is in that direction.

	\begin{Corollary}
		Consider the system (\ref{diff_sys}) and a function $\phi \in C\left[\mathbb{R}_+ \times \mathbb{R}^n\times \mathbb{R}^n,~ \mathbb{R}^n\right]$, $(t, u,x) \mapsto \phi(t, u,x)$, which is quasi-monotone non-decreasing in $u \in \mathbb{R}^n$. Suppose for any solution $\delta \psi(t,x_0,\delta x_0)$ on $t \ge t_0$ of the system (\ref{diff_sys}), 
		\begin{align*} 
		& 2 A~ \text{dvec}(\text{diag}(\delta \psi) ~ \text{diag}(h(\delta \psi, \psi, t))) \\ 
		& ~<~ \phi\left(t, A~ \text{dvec}(\text{diag}(\delta \psi )^2), \psi \right) ~\text{ for all }~ t \ge t_0,   
		\end{align*}
		for a matrix $A = (a_{ij})_{n \times n}$ with $a_{ij} \ge 0$ for all $i, j = 1, 2, \cdots, n$. \\ 
		
		Further, let for $t \ge t_0 > 0$ there exists
		a maximal solution
		$R(t, u_{0},x_{0})$ of \[\dot u = \phi(t,u,x),~ u(t_0) = u_0 \ge 0.\]
		Then, any solution $\delta \psi(t,x_0,\delta x_0)$ of (\ref{diff_sys}) on $t\ge t_0$ with   $A ~ \text{dvec}(\text{diag}(\delta x_0)^2) < u_0$ has the property that
		\[A ~ \text{dvec}(\text{diag}(\delta \psi )^2) < R(t) ~\text{ for all }~ t \ge t_0. \]	
	\end{Corollary}
	\begin{proof}
		The proof is similar to that of Theorem \ref{theorem1}.
	\end{proof}
	\begin{theorem}
		Let $K$ be a pointed closed convex cone in $\mathbb{R}^n$. Consider the system (\ref{diff_sys}) and a function $\phi \in C\left[\mathbb{R}_+ \times \mathbb{R}^n,~ \mathbb{R}^n\right]$, $(t, u) \mapsto \phi(t, u)$, which is quasi-monotone non-decreasing in $ u \in \mathbb{R}^n$ with respect to $K$. \\
		Suppose for any solution $\delta \psi(t,\delta x_0)$ on $t \ge t_0$ of the variational system (\ref{diff_sys}),
		\begin{align*} 
		& 2 A~ \text{dvec}(\text{diag}(\delta \psi) ~ \text{diag}(h(\delta \psi, t))) \\ 
		& ~<_K~ \phi\left(t, A~ \text{dvec}(\text{diag}(\delta \psi )^2) \right) ~\text{ for all }~ t \ge t_0,   
		\end{align*}
		for a matrix $A = (a_{ij})_{n \times n}$ with $a_{ij} \ge 0$ for all $i, j = 1, 2, \cdots, n$. 
		Further, let for $t \ge t_0 > 0$ there exists a maximal solution $R(t)$ of (\ref{comparison}). 

		Then, any solution $\delta \psi(t,\delta x_0)$ of (\ref{diff_sys}) on $t\ge t_0$ with   $A ~ \text{dvec}(\text{diag}(\delta x_0)^2) <_K u_0$ has the property that
		\begin{equation}\label{abc}
		A ~ \text{dvec}(\text{diag}(\delta \psi )^2) <_K R(t) ~\text{ for all }~ t \ge t_0. 
		\end{equation}
		From (\ref{abc}), a similar conclusion as in Theorem \ref{theorem1} is also followed here in the framework of cone.
	\end{theorem}
	\begin{proof}
		The proof is similar to that of Theorem \ref{theorem1}. It is exactly same till the equation  (\ref{S_construction}). Then the rest of the part is appropriate modification of the inequalities w.r.t. the partial ordering induced by $K$. 
	\end{proof}
	We now again consider the system (\ref{diff_syst})
	and assume that it has a finite equilibrium solution $\bar x$. Suppose the squared vector distance of a solution $x$, of the system, from $\bar x$ is given by 
	\[\|x - \bar x\|_v^2 = A ~ \text{dvec}(\text{diag} (x - \bar x)^2), \]
	where $A$ is a real matrix $(a_{ij})_{n\times n}$ where all $a_{ij}$'s are nonnegative.  Obviously, $\|x - \bar x\|_v^2$ is a vector in $\mathbb{R}^n$. In the following, we denote the $i$th row of $A$ by $a_i^T$. Suppose $C \in \mathbb{R}^n$ be the squared vector distance between the initial data $x_0$ and the equilibrium solution $\bar x$. We denote $\|\delta x\|_v^2$ for the squared virtual displacement of $x$ from $\bar x$. Then, evidently, $\|\delta x_0\|^2_v = C$. 
	
	From (\ref{diff_syst}), we get the following differential relation
	\begin{align*}
	\frac{d}{dt}\left( \|\delta x\|^2_v\right) 
	& = 
	\frac{d}{dt} \left[
	\begin{array}{c}
	a_1^T ~ \text{dvec}(\text{diag} (\delta x)^2) \\
	a_2^T ~ \text{dvec}(\text{diag} (\delta  x)^2) \\
	\vdots \\
	a_n^T ~ \text{dvec}(\text{diag} (\delta x)^2)
	\end{array}
	\right] \\ 
	& = 2 ~ A ~ \text{dvec}(\text{diag}(\delta x) ~ \text{diag}(\delta \dot x)).    
	\end{align*}
	Therefore under the assumption in Theorem \ref{theorem1}, we have,

	\begin{equation}\label{ab}
	\frac{d}{dt} \left( \|\delta x\|^2_v  \right) < \phi(\|\delta x\|^2_v) 
	\end{equation}
	For finding the properties of solution of the above inequality, we take the comparison system\\
	\begin{equation}
	\dot u=\phi(t,u),~~~~~~ u(t_{0})=u_{0}>\|\delta x_{0}\|^{2}_{v}=C
	\end{equation}
	Further, if $R(t)$ is the maximal solution of the above equation, then a solution of (\ref{ab}), follows from Theorem \ref{theorem1} satisfies,
	$\|\delta x(t)\|^{2}_{v}<R(t)$.
	If $R(t)$ is exponentially convergent, by component wise integration, we have
	\begin{equation}\label{vector_contract_eq}
	\|\delta x\|_v < C^{\tfrac{1}{2}} \exp \left(-\lambda t\right)
	\end{equation}
	where $\lambda$ is the convergence rate. Then (\ref{vector_contract_eq}) shows that the virtual vector distance $\|\delta x\|_v$ is lesser than $C^{\tfrac{1}{2}}$ and it converges exponentially to zero as $t \rightarrow \infty$,  which implies that all trajectories converges to the equilibrium point $\bar x$ and as a result, the system is asymptotically stable.

	
	\section{Examples}
	We present example 1 and example 2 to illustrate the implementation of Theorem \ref{theorem1} in order to relax the condition of proving largest eigenvalue of the Jacobian to be negative definite. An example 1 is the $n$ dimensional linear system, so it is very tedious to prove the largest eigenvalue of its Jacobian to be negative definite. Also for nonlinear systems it is desirable to use vector contraction analysis as it is again very difficult to prove the largest eigenvalue of the Jacobian to be negative definite which can be easily shown in example 2. Moreover, example 3 shows that the vector valued function is quasimonotone nondecreasing with respect to cone while it is not quasimonotone nondecreasing with respect to the usual  componentwise ordering in $\mathbb{R}^n$.
	\begin{example}\label{example1}
		Consider the system of differential equation
		\begin{align}\label{sy1}
		\begin{split}
		\frac{d}{dt}x_{i}=&-\rho_{i}x_{i}+\sigma ,\ \ i=1,2,\dots,n\\
		\frac{d}{dt}\sigma=&\sum_{i=1}^{n}a_{i}x_{i}-(p+1)\sigma
		\end{split}
		\end{align}
		where  $\rho_{i}>0,\ p>0 $.
		The vector valued norm defined by (\ref{vect-norm}) of the distance between any pair of trajectories for the whole system assuming the matrix $A$ as diagonal matrix with all diagonal entries 1 is defined as
		\begin{equation}
		\|\delta x\|^{2}_{v}=[\delta x_{1}^2,\delta x_{2}^{2}, \dots, \delta x_{n}^{2},\delta x_{n+1}^{2}].\\
		\end{equation}
		The virtual dynamics of the system (\ref{sy1}) becomes
		\begin{align}\label{sy2}
		\begin{split}
		\delta\dot x_{i}=&-\rho_{i}\delta x_{i}+\delta\sigma \\
		\delta\dot\sigma=&\big[\  \sum_{i=1}^{n}a_{i}\delta x_{i}-(p+1)\delta\sigma \ \big]\
		\end{split}
		\end{align}
		The rate of change of squared distance between trajectories for the whole system and using eqn. (\ref{sy2}) is given by
		\begin{align}\label{sy3}
		\frac{d}{dt}(\delta x_{i}^{T}\delta x_{i})&= 2\delta x_{i}^{T}\delta \dot x_{i}\nonumber\\
		& =-2\rho_{i}\delta x_{i}^{T}\delta x_{i}+2\delta x_{i}^{T}\delta\sigma
		\end{align}
		Using the inequality\\
		\begin{equation}\label{ineq}
		\delta x_{i}^{T}\delta\sigma\leq\frac{\rho_{i}\delta x_{i}^{T}\delta x_{i}}{2}+\frac{\delta\sigma^{2}}{2\rho_{i}}\\
		\end{equation}
		Equation (\ref{sy3}) becomes
		\begin{equation}\label{sy21}
		\frac{d}{dt}(\delta x_{i}^{T}\delta x_{i})\leq -\rho_{i}\delta x_{i}^{T}\delta x_{i}+ \frac{\delta \sigma^{2}}{\rho_{i}}
		\end{equation}
		Also, \begin{align}
		\frac{d}{dt}(\delta\sigma^{2}) 
		&=\sum_{i=1}^{n} 2a_{i}\delta\sigma\delta x_{i}-2(p+1)\delta\sigma^{2}
		\end{align}
		Using the inequality (\ref{ineq}) again, we have
		\begin{align}\label{sy23}
		\frac{d}{dt}(\delta\sigma^{2})
		\leq&\sum_{i=1}^{n}|a_{i}|\rho_{i}\delta x_{i}^{T}\delta x_{i}-\nonumber\\
		&\Big(2(p+1)-\sum_{i=1}^{n}\frac{|a_{i}|}{\rho_{i}}\Big)\delta\sigma^{2}
		\end{align}
		With the help of (\ref{sy21}) and (\ref{sy23}), we consider the following comparison system
		\begin{align}\label{comp}
		\begin{split}
		\frac{d}{dt}w_{i}=&-\rho_{i}w_{i}+\frac{w_{n+1}}{\rho_{i}}\\
		\frac{d}{dt}w_{n+1}=&\sum_{i=1}^{n}|a_{i}|\rho_{i}w_{i}-\Big(2(p+1)-\sum_{i=1}^{n}\frac{|a_{i}|}{\rho_{i}}\Big) w_{n+1}
		\end{split}
		\end{align}
		The system (\ref{comp}) is quasimonotone nondecreasing in $w$ and also convergent if~ $2(p+1)>\sum_{i=1}^{n}\frac{|a_{i}|}{\rho_{i}}$. Moreover, the equilibrium point of the original system is zero. Therefore, the original system is asymptotically stable.
	\end{example}
	\begin{example}
		Consider a nonlinear differential system
		\begin{align}\label{nonlinear}
		\begin{split}
		\dot x_{1}=&-x_{1}^{2}+x_{2}\\
		\dot x_{2}=& x_{1}-2x_{2}^{2}
		\end{split}
		\end{align}
		Taking the virtual dynamics of the above system, we get
		\begin{align}
		\begin{split}
		\delta \dot x_{1}=& -2x_{1}\delta x_{1}+\delta x_{2}\\
		\delta \dot x_{2}=&\delta x_{1}-4x_{2}\delta x_{2}
		\end{split}
		\end{align}	
		The vector valued norm of the distance between any pair of trajectories of the whole system assuming $A$ as a diagonal matrix with all diagonal entries 1 as
		$$\|\delta x\|^{2}_{v}=\left[
		\begin{array}{c}
		\delta x_{1}^2\\
		\delta x_{2}^2
		\end{array}
		\right]$$
		The rate of change of this vector valued norm can be obtained as
		\begin{align*}
		\begin{split}
		\frac{d}{dt}(\delta x_{1}^2)=
		&-4x_{1}\delta x_{1}^2+2\delta x_{1}\delta x_{2}\\
		&\leq (1-4x_{1})\delta x_{1}^2+\delta x_{2}^2\\
		\frac{d}{dt}(\delta x_{2}^2)=
		&~\delta x_{1}^2+\delta x_2^2-8x_{2}\delta x_{2}^2\\
		&\leq(1-8x_{2})\delta x_{2}^2 +\delta x_{1}^2
		\end{split}
		\end{align*}
		Therefore, the comparison system of (\ref{nonlinear})
		\begin{align}
		\begin{split}
		\dot w_{1}=(1-4x_{1})w_{1}+w_{2}\\
		\dot w_{2}=(1-8x_{2})w_{2}+w_{1}
		\end{split}
		\end{align}
		is quasimonotone nondecreasing in $w$ and also it is convergent if $x_{1}>\frac{1}{4},x_{2}>\frac{1}{8}$. Therefore, from the Theorem 2,
		$\text{dvec}(\text{diag}(\psi(t))^2)<w(t)$ taking $\text{dvec}(\text{diag}(\delta x_0)^2) <w_{0}$.
		The simulation results are shown in Fig.1. 
	\end{example}
	\begin{figure}[h] 
		\includegraphics[width=1.5in]{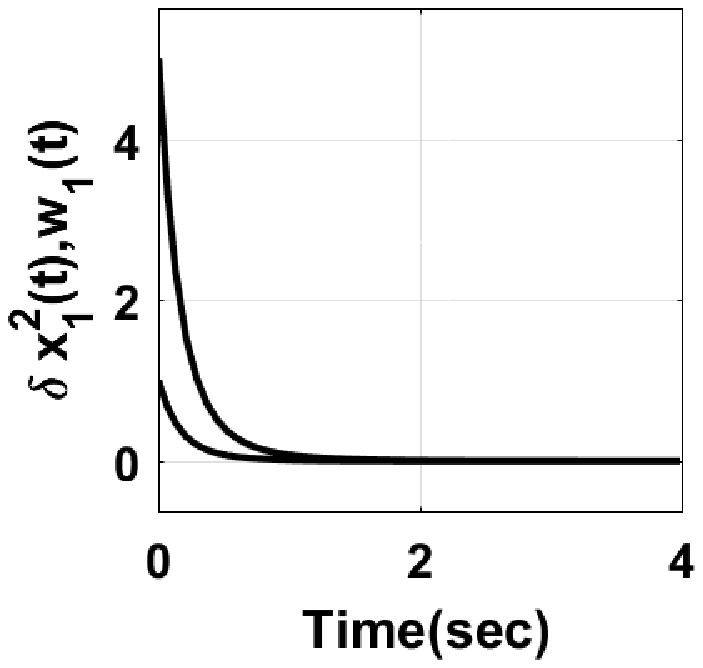}
		\includegraphics[width=1.5in]{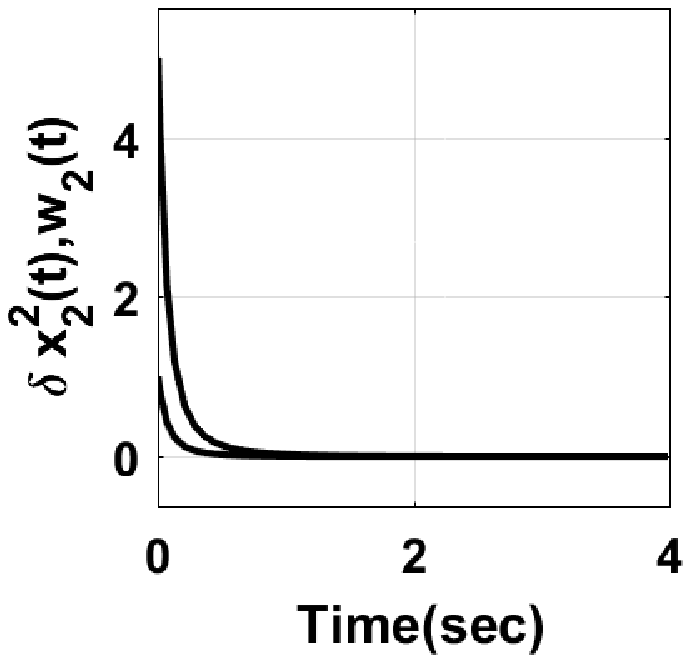}
		\caption{Response of $\text{dvec}(\text{diag}(\delta x_{1}(t))^2),\text{dvec}(\text{diag}(\delta x_{2}(t))^2)$, $w_{1}(t),w_{2}(t)$ with initial states $\text{dvec}(\text{diag}(\delta x_{1}(0))^2) = 1$, $\text{dvec}(\text{diag}(\delta x_{2}(0))^2)=1$ and $w_{1}(0)= 5, w_{2}(0)=5$}
	\end{figure}
	\begin{example}
		In this example, we consider the problem in Example \ref{example1} for $n=1$ and without the assumption $\rho_{1}\geq 0,\rho_{2}\geq 0$. Then, the function $F(w_{1},w_{2})$ of the right hand side of the system of equation of the problem :
		\begin{align}
		\begin{split}
		\frac{d}{dt}w_{1}=& -\rho_{1}w_{1}+\frac{w_{2}}{\rho_{1}}\\
		\frac{d}{dt}w_{2}=& |a_{1}|w_{1}-(2p-\frac{|a_{1}|}{\rho_{1}}-\frac{|a_{2}|}{\rho_{2}})w_{2}
		\end{split}
		\end{align}
		is not quasimonotone nondecreasing since the coefficient of $w_{2}$ may possibly be negative. For instance, if we take $a_{1}=1, a_{2}=0,\rho_{1}=-\frac{1}{2},\rho_{2}=\frac{1}{2}$ and $p=1$. The function on the right hand side of (21) becomes
		\begin{align*}
		F(w_{1},w_{2})=
		\begin{bmatrix}
		\frac{1}{2}w_{1}-2w_{2}\\
		w_{1}-4w_{2}
		\end{bmatrix}
		\end{align*}
		Note that this F is not quasimonotone nondecreasing w.r.t usual componentwise ordering in $\mathbb R^{2}$.
		However, F is quasimonotone nondecreasing w.r.t the cone $K=\{ (w_{1},w_{2})\in \mathbb{R}_{+}^{2}|w_{2}\leq w_{1}\leq 3w_{2}\}$.
		Since,
		\begin{enumerate}[(i)]
			\item  On the boundary $w_{1}=w_{2},$ taking $\phi=(1,-1)\in K^{*},$ 
			$\langle\phi,(w_{1},w_{1})\rangle=\langle(1,-1),(w_{1},w_{1})\rangle=0$\\
			and $\langle\phi,F(w_{1},w_{2})\rangle=\langle(1,-1),(\frac{1}{2}w_{1}-2w_{1},w_{1}-4w_{1})\rangle\\
			=\frac{3}{2}w_{1}\geq 0, ~ ~\forall w_{1}\geq 0$
			\item On the boundary $w_{1}=3w_{2}$,
			taking $\phi=(1,-3),$ 
			$\langle\phi,(3w_{2},w_{2})\rangle=0$ and $\langle\phi,F(3w_{2},w_{2})\rangle=1(\frac{3}{2}w_{2}-2w_{2})-3(3w_{2}-4w_{2})\\
			=\frac{5}{2}w_{2}\geq 0,~ ~\forall w_{2}\geq 0.$
		\end{enumerate}
	\end{example}
	\section{Conclusion}
	A generalized vector contraction analysis framework utilizing the notion of vector distances for addressing convergence of trajectories of nonlinear dynamical systems is presented in this paper. In particular, we derived comparison results employing the quasi-monotonicity property of the function for proving convergence of the original dynamical system by comparing the solutions of the auxiliary system and the original system. Furthermore, in order to overcome the componentwise inequalities of vectors, the results are also derived in the framework of the cone ordering. Illustration of the derived results is presented by examples. Finally, from the derived results, it is able to show the convergence of the nonlinear system by proving the  convergence of the comparison system without much less strict conditions.


\begin{center}
	\textit{Appendix}
\end{center}
\textit{Proof of Theorem 1.}\\
\textit{Proof of (a) Norm Property}.
\begin{enumerate}[(i)]
	\item This property follows from the fact that for each $i = 1, 2, \cdots, m$, $D_i(\delta x) = 0 \Longleftrightarrow (\delta x_1, \delta x_2, \cdots, \delta x_n) = (0, 0, \cdots, 0).$
	\item From the expression of $D_i(\delta x)$, we note that $D_i(c~\delta x) = |c| D_i(\delta x)$ for each $i = 1, 2, \cdots, m$.  Hence, $\|c\delta x\|_v = |c|~\|\delta x\|_v$.
	\item To prove this property, it suffice to prove that for each $i = 1, 2, \cdots, m$, the $i$-th component $\|\delta x\|_v$ has the property that
	$D_i(\delta x + \delta y) \le D_i(\delta x) + D_i(\delta y) ~\text{ for all }~ \delta x, \delta y \in \mathbb{R}^n.$
	Since every $a_{ij} \ge 0$, we get from the Cauchy-Schwarz inequality that
	\begin{align*}\label{cs}
	\sum_{j = 1}^n a_{ij} |\delta x_j||\delta y_j|~&~ = \sum_{j = 1}^n \left|\sqrt{a_{ij}} ~ \delta x_j\right|~\left|\sqrt{a_{ij}} ~ \delta y_j\right|\\
	~&~\le \sqrt{ \sum_{j = 1}^n \sqrt{a_{ij}} ~ \delta x_j^2 }
	\sqrt{ \sum_{j = 1}^n
		\sqrt{a_{ij}} ~ \delta y_j^2}.
	\end{align*}
	\begin{align*}
	(D_i(\delta x + \delta y))^2 ~&~ = ~ \sum_{j = 1}^n a_{ij} ~ (\delta x_j + \delta y_j)^2 \\
	~&~ \le ~ \sum_{j = 1}^n a_{ij} ~ \left( \delta x_j^2 + 2 \delta x_j \delta y_j + \delta y_j^2 \right) \\
	~&~ \le ~ \sum_{j = 1}^n a_{ij} ~ \delta x_j^2 + 2 ~ \sum_{j = 1}^n a_{ij} ~ |\delta x_j| ~ |\delta y_j| +\\
	~&~~~~~~ \sum_{j = 1}^n a_{ij} ~ \delta y_j^2 \\
	~&~ \le ~ \left(\sqrt{\sum_{j = 1}^n a_{ij} ~ \delta x_j^2} + \sqrt{\sum_{j = 1}^n a_{ij} ~ \delta y_j^2} \right)^2.
	\end{align*}
\end{enumerate}
Therefore, $D_i(\delta x + \delta y) \le D_i(\delta x) + D_i(\delta y)$.\\
\textit{Proof of (b) Convexity property}.
The proof is followed from the fact that the $i^{th}$ component function $D_i(\delta x)$, being a norm in $\mathbb{R}^n$, a convex function on $\mathbb{R}$, for each $i = 1, 2, \cdots, n$.\\
\textit{Proof of (c) Locally-Lipschitzian property}.
For any $i \in \{1, 2, \cdots, n\}$, the $i$-th component function of $F(\delta x)$ is
\[(D_i(\delta x))^2 = \sum_{j = 1}^n a_{ij} \delta x_j^2.  \]
Note that the Hessian matrix of $(D_i(\delta x))^2$ is the diagonal matrix diag$(a_{i1}, a_{i2}, \cdots, a_{im})$ which is positive semi-definite as every $a_{ij} \ge 0$. Hence, $(D_i(\delta x))^2$ is a convex function on $\mathbb{R}^n$.
By the result in \cite{way1972every}, for each $i = 1, 2, \cdots, m$, there exists a constant $k_i > 0$ such that
$ \left|(D_i(\delta x))^2 - (D_i(\delta y))^2 \right| \le k_i \|\delta x -\delta y \| ~\text{ for all }~ \delta x, \delta y \in  C.$
Therefore, for any $\delta x, \delta y$ in $C$,
\begin{align*}
\left\| F(\delta x) - F(\delta y) \right\| & ~=~ \sqrt{\sum_{j = 1}^n \left((D_1(\delta x))^2 - (D_1(\delta y))^2\right)^2} \\
& ~=~ \sqrt{k_1^2 + k_2^2 + \cdots + k_m^2} ~\|\delta  x - \delta y\|.
\end{align*}
The result follows by letting $k = \sqrt{k_1^2 + k_2^2 + \cdots + k_m^2}$.\\
\textit{Proof of (d) Differentiability}.
Let $h = (h_1, h_2, \cdots, h_n)^T$. The result is followed by the following limit:
\begin{align*}
& \lim_{h \rightarrow 0} \frac{1}{\|h\|} \left\| F(\delta x+h) - F(\delta x) - 2 ~A~ \delta x~ h\right\|  \\
= ~&~ \lim_{h \rightarrow 0} \frac{1}{\|h\|} \left\| A(\delta x+h)^2 - A(\delta x)^2 - 2 ~A~ \delta x~ h\right\| \\
= ~&~ \lim_{h \rightarrow 0} \frac{1}{\|h\|} \left\| A \left[ h_1^2, h_2^2, \cdots, h_n^2\right]^T \right \| \\
= ~&~ \lim_{h \rightarrow 0}  \left\| A \left[ \frac{h_1^2}{\|h\|}, \frac{h_2^2}{\|h\|}, \cdots, \frac{h_n^2}{\|h\|}\right]^T \right \|\\
= ~&~ 0.
\end{align*}


\end{document}